\documentclass[11 pt,reqno]{article}

\usepackage{amsmath}
\usepackage{amstext}
\usepackage{enumerate}
\usepackage{amsfonts}
\usepackage{amscd}
\usepackage{amssymb}
\usepackage{amsthm}
\usepackage{tikz-cd}

\usepackage[top=1in, bottom=1in, left=1.5in, right=1.5in]{geometry}

\usepackage[hidelinks]{hyperref}

\newcommand{\Hom}{\operatorname{Hom}}

\newcommand{\Spec}{\operatorname{Spec}}
\newcommand{\colim}{\operatorname{colim}}

\newcommand{\Sm}{\operatorname{Sm}}
\newcommand{\CH}{\operatorname{CH}}

\newcommand{\R}{\mathbb{R}}
\newcommand{\F}{\mathbb{F}}
\newcommand{\Z}{\mathbb{Z}}
\newcommand{\bQ}{\mathbb{bQ}}

\newcommand{\A}{\mathbb{A}}

\newcommand{\DM}{\mathcal{DM}}
\newcommand{\SH}{\mathcal{SH}}
\newcommand{\DMT}{\mathcal{DMT}}

\newcommand{\oH}{\operatorname{H}}
\newcommand{\oM}{\operatorname{M}}

\theoremstyle{plain}
  \newtheorem{thm}{Theorem}
  \newtheorem*{thm*}{Theorem}
  \newtheorem{prop}[thm]{Proposition}
  \newtheorem{lem}[thm]{Lemma}
  \newtheorem{cor}[thm]{Corollary}
  \newtheorem*{cor*}{Corollary}
\theoremstyle{definition}
  \newtheorem{rem}[thm]{Remark}

\theoremstyle{remark}
  
  \newtheorem*{defn*}{Definition}
  
  \newtheorem*{conjecture*}{Conjecture}
  
  \newtheorem*{note*}{Notation}
  \newtheorem{question}[thm]{Question}
  \newtheorem*{question*}{Question}

\title{Detecting motivic equivalences with motivic homology}
\author{David Hemminger}

\newcommand{\address}{{
  \small

  \textsc{UCLA Mathematics Department, Box 951555, Los Angeles, CA 90095-1555}\par\nopagebreak
  \texttt{dhemminger@math.ucla.edu}
}}

\begin{document}
\maketitle
Let $k$ be a field, let $R$ be a commutative ring, and assume the exponential characteristic of $k$ is invertible in $R$.  Let $\DM(k; R)$ denote Voevodsky's triangulated category of motives over $k$ with coefficients in $R$.

In a failed analogy with topology, motivic homology groups do not detect isomorphisms in $\DM(k; R)$ (see Section \ref{sec:background}).  However, it is often possible to work in a context partially agnostic to the base field $k$.  In this note, we prove a detection result for those circumstances.

\begin{thm}\label{thm:main}
Let $\varphi \colon M \to N$ be a morphism in $\DM(k; R)$.  Suppose that either
\begin{enumerate}[a)]
\item 
For every separable finitely generated field extension $F/k$, the induced map on motivic homology $\oH_*(M_F,R(*)) \to \oH_*(N_F,R(*))$ is an isomorphism, or
\item
Both $M$ and $N$ are compact, and for every separable finitely generated field extension $F/k$, the induced map on motivic cohomology $\oH^*(N_F,R(*)) \to \oH^*(M_F,R(*))$ is an isomorphism.
\end{enumerate}
Then $\varphi$ is an isomorphism.
\end{thm}


Note that if $X$ is a separated scheme of finite type over $k$, then the motive $\oM(X)$ of $X$ in $\DM(k;R)$ is compact, by \cite[Theorem 11.1.13]{CD19}.

The proof is easy: following a suggestion of Bachmann, we reduce to the fact (Proposition \ref{prop:spectra-iso-detection}) that a morphism in $\SH(k)$ which induces an isomorphism on homotopy sheaves evaluated on all separable finitely generated field extensions must be an isomorphism.  This follows immediately from Morel's work on unramified presheaves (see \cite{Morel12}).

Together with existing results, Theorem \ref{thm:main} detects isomorphisms between certain spaces in the pointed motivic homotopy category $\mathcal{H}(k)_*$.  Wickelgren-Williams showed that the functor 
$$\Sigma_{S^1}^\infty \colon \mathcal{H}(k)_*\to \mathcal{SH}^{S^1}(k)$$
is conservative when restricted to the subcategory of $A^1$-1-connected spaces (\cite[Corollary 2.23]{WW19}).  Bachmann proved that the functor
$$\Sigma_{\mathbb{G}_m}^\infty \colon \mathcal{SH}^{S^1}(k)_{\ge 0}\cap \mathcal{SH}^{S^1}(k)(n) \to \mathcal{SH}(k)$$
is conservative when $\operatorname{char}(k) = 0$ and $n = 1$ or $k$ is perfect and $n = 3$ (\cite[Corollary 4.15]{Bachmann20}), and Feld showed that $\Sigma_{\mathbb{G}_m}^\infty$ is conservative when $k$ is an infinite perfect field of characteristic not 2 (\cite[Theorem 6]{Feld20}).  Here $\mathcal{SH}^{S^1}(k)(n) \subseteq \mathcal{SH}^{S^1}(k)$ denotes the full subcategory generated under homotopy colimits by the objects $X_+ \wedge \mathbb{G}_m^{\wedge n}$, for $X$ a smooth scheme over $k$.  By another theorem of Bachmann, the functor
$$\oM \colon \mathcal{SH}^{c}(k)\to \DM(k;\Z)$$
is conservative when $k$ is a perfect field of finite 2-{\'e}tale cohomological dimension, where $\mathcal{SH}^c(k) \subseteq \mathcal{SH}(k)$ denotes the full subcategory of compact objects (\cite[Theorem 16]{Bachmann18}).  Using Bachmann's results, Totaro showed that if $k$ is a finitely generated field of characteristic zero, $X \in \mathcal{SH}(k)^c$, $\oM(X) = 0$ in $\DM(k,\Z)$, and $\oH_*(X(\R),\Z[1/2]) = 0$ for every embedding of $k$ into $\R$, then $X = 0$ (\cite[Theorem 6.1]{Totaro20}).

Finally, $\DM(k;\Z[1/p])$ is equivalent to the category with objects the same as $\DM(k;\Z)$ and hom-sets tensored by $\Z[1/p]$, by \cite[Proposition 11.1.5]{CD19}.  It follows that if $\mathcal{C}$ is a full subcategory of $\mathcal{H}(k)_*$ such that the composition of functors $\mathcal{C}\subseteq \mathcal{H}(k)_* \to \DM(k;\Z)$ is conservative, then the composition
$$\mathcal{H}(k)_*[1/p]\cap \mathcal{C} \subseteq \mathcal{H}(k)_* \to \DM(k;\Z) \to \DM(k;\Z[1/p])$$
is conservative.  Altogether, we have the following corollary. 

\begin{cor}\label{cor:main}
Let $\mathcal{H}(k)_*(n), \mathcal{H}(k)_{*,\ge 1}, \mathcal{H}(k)_*^c \subseteq \mathcal{H}(k)_*$ denote the subcategory generated under homotopy colimits by the objects $X_+\wedge \mathbb{G}_m^{\wedge n}$, for $X$ a smooth scheme over $k$; the subcategory of $\A^1$-1-connected spaces; and the subcategory of objects sent to compact objects in $\mathcal{SH}(k)$ under stabilization, respectively.  Let $k$ be a perfect field with exponential characteristic $p$.  Let $n = 1$ if $p = 1$, $n = 2$ if $p > 2$ and $k$ is infinite, and $n = 3$ otherwise.  Let $\varphi \colon X \to Y$ be a morphism in 
$$\mathcal{H}(k)_*(n)\cap \mathcal{H}(k)_{*,\ge 1} \cap \mathcal{H}(k)_*^c \cap \mathcal{H}(k)_*[1/p].$$
Suppose that either
\begin{enumerate}[a)]
\item 
$k$ has finite 2-{\'e}tale cohomological dimension, or
\item
$k$ is a finitely generated field of characteristic zero, and for every embedding of $k$ into $\R$, the induced map
$$\oH_*(X(\R),\Z[1/2]) \to \oH_*(Y(\R),\Z[1/2])$$
is an isomorphism.
\end{enumerate}
Suppose further that either
\begin{enumerate}[a)]
\item 
for every finitely generated field $F/k$, the induced map on motivic homology $\oH_*(X_F,\Z(*)) \to \oH_*(Y_F,\Z(*))$ is an isomorphism, or
\item
for every finitely generated field $F/k$, the induced map on motivic cohomology $\oH^*(Y_F,\Z(*)) \to \oH^*(X_F,\Z(*))$ is an isomorphism.
\end{enumerate}
Then $\varphi$ is an isomorphism in $\mathcal{H}(k)_*$.
\end{cor}

\begin{rem}
Suppose $X$ is a smooth separated scheme of finite type over $k$.  Then 
$$X_+ \wedge S^{m+n,n} \in \mathcal{H}(k)_*(n)\cap \mathcal{H}(k)_{*,\ge 1} \cap \mathcal{H}(k)_*^c,$$
where $m = 0$ if $X$ is $\A^1$-1-connected, $m = 1$ if $X$ is $\A^1$-connected, and $m = 2$ otherwise.  Localization away from a prime $p$ is well-behaved on $\A^1$-1-connected spaces; see \cite{AFH19}.  Perfect fields of finite 2-{\'e}tale cohomological dimension include algebraically closed fields, finite fields, and totally imaginary number fields.
\end{rem}

\section{The derived category of motives}\label{sec:background}
We recall here a few basic facts about Voevodsky's derived category of motives $\DM(k;R)$.  A more thorough, but still concise, summary of basic properties of $\DM(k;R)$ is \cite[Section 5]{Totaro16}.  Further references on $\DM(k;R)$ are \cite{Voevodsky00} and \cite{CD19}.

The category $\DM(k;R)$ is defined in \cite[Definition 11.1.1]{CD19}.  It is tensor triangulated with tensor unit $R$, and it has arbitrary (not necessarily finite) direct sums.  The motive and the motive with compact support define covariant functors $\oM$ and $\oM^c$ from separated schemes of finite type over $k$ to $\DM(k;R)$.  We write $R(i)[2i]$ for the \emph{Tate motive} $\oM^c(A^i_k)$.  The object $R(1)$ is invertible in $\DM(k;R)$; we define $R(-i) = R(i)^*$.  For $M \in \DM(k;R)$, we write $\oM(i) = M \otimes R(i)$.

For any motive $M \in \DM(k)$, the functor $M \otimes -$ has a right adjoint $\underline{\Hom}(M,-)$.  By \cite[Lemma 5.5]{Totaro16}, $\underline{\Hom}(M,N) \cong M^* \otimes N$ and $(M^*)^* \cong M$ whenever $M$ is compact.  If $X$ is a smooth scheme over $k$ of pure dimension $n$, a version of Poincar{\'e} duality says that $\oM(X)^* \cong \oM^c(X)(-n)[-2n]$.

For $M \in \DM(k;R)$, the motivic homology and motivic cohomology of $M$ are defined by
$$\oH_j(M,R(i)) = \Hom(R(i)[j],M)$$
and
$$\oH^j(M,R(i)) = \Hom(M,R(i)[j]).$$
Motivic (co)homology does not detect isomorphisms in $\DM(k;R)$, by the following argument (see also \cite[Section 7]{Totaro16}).  Let $\DMT(k;R)$ denote the localizing subcategory of mixed Tate motives, i.e.\ the smallest localizing subcategory containing the motives $R(i)$ for all integers $i$.  Let $C \colon \DM(k;R) \to \DM(k;R)$ denote the composition of the inclusion $\DMT(k;R) \to \DM(k;R)$ with its right adjoint.  For $M \in \DM(k;R)$, the counit morphism $C(M) \to M$ always induces an isomorphism on motivic homology, but is only an isomorphism if $M \in \DMT(k;R)$.

Let $HR\in \SH(k)$ denote the Eilenberg-Maclane spectrum of $R$.  The category $HR\operatorname{-Mod}$ of $HR$-module spectra in $\SH(k)$ is equivalent to $\DM(k;R)$ by \cite[Theorem 1.1]{RO08}.  There is a functor $\SH(k) \to \DM(k;R)$ that can be viewed as smashing with $HR$; it is left adjoint to the forgetful functor $H\colon \DM(k;R) \to \SH(k)$.  The functor $\oM \colon \Sm_k \to \DM(k;R)$ factors through $\SH(k)$ as $(HR\wedge -) \circ \Sigma^\infty$.

Motivic homology does detect isomorphisms in $\DMT(k;R)$.  It is also conjectured (see \cite[1.2(A)]{Beilinson02}) that for $l$ a prime invertible in $k$ the {\'e}tale realization $\DM(k;\bQ) \to \mathcal{D}(\bQ_l)$ is conservative after restricting to the subcategory of compact objects in $\DM(k;\bQ)$.  One could also ask the following related question.

\begin{question}
For a prime $l$ invertible in $k$, is the {\'e}tale realization functor $\DM(k;\F_l)\to \mathcal{D}(\F_l)$ conservative when restricted to compact objects?  Equivalently, is there a nonzero compact motive $M \in \DM(k;\F_l)$ with zero {\'e}tale cohomology?
\end{question}

\section{Proof of Theorem \ref{thm:main}}\label{sec:proof}
We reduce the theorem to the following proposition, which follows immediately from work of Morel on unramified presheaves (see \cite{Morel12}).

\begin{prop}\label{prop:spectra-iso-detection}
Suppose that for a morphism of spectra $\psi\colon E_1 \to E_2$ in $\SH(k)$ the induced map $\pi_{*,*}(E_1)(F) \to \pi_{*,*}(E_2)(F)$ is an isomorphism for all separable finitely generated field extensions $F/k$.  Then $\psi$ is an isomorphism in $\SH(k)$.
\end{prop}

\begin{proof}
A presheaf of sets $S$ on $\Sm_k$ is \emph{unramified} if for any $X \in \Sm_k$, $S(X) \cong \prod S(X_\alpha)$, where the $X_\alpha$ are the irreducible components of $X$, and for any dense open subscheme $U \subset X$, the restriction map $S(X) \to S(U)$ is injective and moreover an isomorphism if $X-U$ has codimension at least 2 everywhere.  For a space $Y \in \mathcal{H}(k)_*$, the unstable homotopy sheaf $\pi_{n}^{\A^1}(Y)$ is unramified for $n\ge 2$, by \cite[Theorem 9, Example 1.3]{Morel12} (see also \cite[Remark 17]{Morel12}).

For $E \in \SH(k)$, 
$$\pi_{i,j}(E)\cong \pi_{2}^{\A^1}(\Omega^\infty(S^{2-i,-j}\wedge E)),$$
so it is unramified.  In particular, for $X$ a smooth variety over $k$, $\pi_{i,j}(E)(X)$ injects into $\pi_{i,j}(E)(\Spec(k(X)))$.  (If $S$ is a presheaf on $\Sm_k$ and $Y \cong \lim_\alpha Y_\alpha$ in the category of schemes over $k$ with $Y_\alpha \in \Sm_k$, then $S(Y) = \colim_\alpha(Y_\alpha)$ by definition.)   It follows that the unstable homotopy sheaves of the cone of $\psi$ vanish, so $\psi$ is an isomorphism.
\end{proof}

If $M \in \DM(k;R)$ is compact, then $\Hom(M,R(i)[j]) \cong \Hom(R(-i)[-j],M^*)$ and $(M^*)^* \cong M$, so it suffices to prove the theorem with the assumption (a).  The functor $H$ is conservative, so it suffices to show that $HC$ is zero, where $C$ is the cone of $\varphi$.  By the hypotheses of the theorem we have that $\oH_*(C_{k(X)},R(*)) = 0$, so the result follows from Proposition \ref{prop:spectra-iso-detection} and the following lemma.

\begin{lem}
Let $M \in \DM(k;R)$, and let $X$ be a smooth variety over $k$.  Then 
$$\pi_{i,j}(HM)(\Spec(k(X)))\cong \oH_{i+2n}(M_{k(X)},R(j+n)),$$
where $n = \dim X$.
\end{lem}
\begin{proof}
We have that 
$$\pi_{i,j}(HM)(\Spec(k(X))) = \colim_S \pi_{i,j}(HM)(X - S)$$
by definition, where the colimit ranges over all closed subschemes $S\subsetneq X$.  Sheafification does not change values on stalks, so by Poincar{\'e} duality and the fact that $\oM(X-S)$ is compact, 
$$\colim_S \pi_{i,j}(HM)(X - S)\cong \colim_S \oH_{j+2n}(\oM^c(X-S)\otimes M,R(i+n)).$$
Thus it suffices to show that
$$\colim_S \oH_*(\oM^c(X - S)\otimes M,R(*))\cong \oH_*(M_{k(X)},R(*))$$

The property of $M$ that both sides are isomorphic is preserved under arbitrary direct sums.  Indeed, if $f \colon \Spec(k(X)) \to \Spec(k)$ is the map induced by the field extension, then $M_{k(X)} = f^*M$.  For any morphism of schemes, the associated pullback functor on derived categories of motives has a right adjoint (\cite[Theorem B.1]{CD19}).  Since $f^*$ and $\oM^c(X - S) \otimes -$ are left adjoints, they commute with arbitrary direct sums.  Motivic homology commutes with arbitrary direct sums because its representing objects $R(i)[j]$ are compact.
    
Moreover, if two motives in an exact triangle satisfy the property, then the third does as well, by the long exact sequence in motivic homology for an exact triangle.  Thus it suffices to prove the lemma when $M = \oM^c(Y)$, where $Y$ is a smooth projective scheme over $k$.  In that case, $M_{k(X)} = \oM^c(Y_{k(X)})$, and the motivic homology groups in the lemma are isomorphic to higher Chow groups by \cite[Proposition 4.2.9]{Voevodsky00}, \cite[Theorem 5.3.14]{Kelly17}, and \cite[Proposition 8.1]{CD15}.  Thus it suffices to show that
\begin{equation}\label{eq:higher-chow}\colim_S \CH^*(Y\times_k (X - S),*;R) \cong \CH^*(Y_{k(X)},*;R)\end{equation}
In fact, (\ref{eq:higher-chow}) holds at the level of cycles:  Let 
$$\Delta^n = \Spec\left(k[x_0,\ldots,x_n]/\left(\sum_{i=0}^{n}x_i = 0\right)\right).$$
For every subset $\{i_0,\ldots,i_m\} \subseteq \{0,\ldots,n\}$, there is an associated {\it face} $\Delta^m \subseteq \Delta^n$.  For a smooth scheme $T$, let $z^*(T,n)$ denote the subgroup of algebraic cycles in $T \times \Delta^n$ generated by the subvarieties which intersect $T \times \Delta^m$ in the expected dimension for each face $\Delta^m \subseteq \Delta^n$.  The higher Chow groups $\CH^*(T,*;R)$ are the homology groups of a complex
$$\ldots \to z^*(T,2)\otimes R \to z^*(T,1)\otimes R \to z^*(T,0)\otimes R \to 0.$$
The isomorphism (\ref{eq:higher-chow}) follows from the fact that 
$$\colim z^*(Y\times_k(X - S),n) \cong z^*(Y_{k(X)},n)$$
for all $n$.
\end{proof}

{\it Acknowledgements.}  I thank Burt Totaro for suggesting something like Theorem \ref{thm:main} might be true and for pointing me to \cite{Totaro16}.  I thank Tom Bachmann for suggesting the use of Morel's work on unramified sheaves, which shortened the proof of Theorem \ref{thm:main} from a prior version.  This work was partially supported by National Science Foundation grant DMS-1701237.

\bibliography{motivic-detection-references}
\bibliographystyle{abbrv}

\address
\end{document}